\documentclass[11pt,reqno]{amsart}

\usepackage{paralist}
\usepackage[utf8]{inputenc}
\usepackage[english]{babel}
\usepackage{amsmath,amsthm}
\usepackage{amssymb}
\usepackage{graphicx}
\usepackage[
  pagebackref=true,
  colorlinks,
]{hyperref}
\usepackage{supertabular}
\usepackage[all,cmtip]{xy}
\usepackage{graphics} 
\usepackage{epsfig} 
  
\hypersetup{urlcolor=blue, citecolor=red}

\newtheorem{thm}{Theorem}[section]
\newtheorem*{conj*}{Conjecture}

\newtheorem{lemma}[thm]{Lemma}
\newtheorem{prop}[thm]{Proposition}
\newtheorem{cor}[thm]{Corollary}

\newtheorem{THM}{Theorem}

\newcounter{qqq}
\vspace{.5cm} 

\newcommand{\Proj}{\mathbb{P}}

\newcommand{\W}{\mathcal{W}}

\newcommand{\C}{\mathbb{C}}
\newcommand{\Z}{\mathbb{Z}}

\newcommand{\Fo}{\mathcal{F}}
\newcommand{\Go}{\mathcal{G}}
\newcommand{\Ho}{\mathcal{H}}

\newcommand{\Ol}{\mathcal{O}}

\newcommand{\Aff}{\mathrm{Aff}(\C^2)}
\newcommand{\PP}{\mathbb{P}}

\DeclareMathOperator{\Sym}{Sym}

\title[Logarithmic affine structures, $d$-webs and normal forms]{Logarithmic affine structures, parallelizable $d$-webs and normal forms}

\author{Ruben Lizarbe}
\address{Universidade do Estado do Rio de Janeiro/UERJ, R. S\~ao Francisco Xavier, 524, Maracan\~a,  20550-900, Rio de Janeiro, Brazil. }
\email{ruben.monje@ime.uerj.br}

\author{Frank Loray} 
\address{Univ Rennes, CNRS, IRMAR, UMR 6625, F-35000 Rennes, France. }
\email{frank.loray@univ-rennes1.fr}

\thanks{The authors acknowledge support from CAPES/COFECUB 
(Ma 932/19 ``Feuilletages holomorphes et int\'eraction avec la g\'eom\'etrie'' / process number 88887.192325/2018-00). 
The first author is grateful to the Institut de Recherche
en Math\'ematique de Rennes, IRMAR and the Universit\'e de Rennes 1 
for their hospitality and support. The second author is supported by CNRS, and ANR-16-CE40-0008 project “Foliage”. 
The authors also thank Brazilian-French Network in Mathematics.}

\begin{document}

\begin{abstract}
We study the local analytic classification of affine structures with logarithmic pole on complex surfaces. 
With this result in hand, we can get the local classification of the logarithmic parallelizable $d$-webs, $d\geq 3$. 
\end{abstract}

\maketitle

\smallskip
\noindent


\section{Introduction}

An affine structure on a (smooth) complex surface $S$ is a maximal atlas of charts
$(\phi_i:U_i\to\C^2)$ with transition charts $\phi_j\circ\phi_i^{-1}$ induced by global
affine transformations 
$$\Aff=\left\{F:\C^2\to\C^2\ ;\ Z\mapsto AZ+B,\ \ \ A\in\mathrm{GL}_2(\C),\ B\in\C^2\right\}$$
(see \cite{Klingler} and references therein). In other words, this is a $(G,X)$-structure
in the sense of Ehresmann-Thurston, where $X=\C^2$ and $G=\Aff$.
The vector space of constant vector fields
$\C\langle\partial_x,\partial_y\rangle$ on $X=\C^2$ is invariant under the action of $\Aff$
and is therefore well defined locally on $S$ when pulled-back by the special charts $\phi_i$'s. 
This defines a flat (or curvature free) and torsion free affine connection $T_S\times T_S\to T_S\ ;\ (X,Y)\mapsto \nabla_XY$
whose local horizontal sections correspond to these vector fields. Flatness provides the existence of a basis 
$(Y_1,Y_2)$ of local horizontal sections, i.e. $\nabla_X Y_i=0$ for all $X$.
Torsion freeness implies commutativity of the corresponding vector fields $[Y_1,Y_2]=0$.
See section \ref{S:AffineRiccatiConnections}
for more details. Klingler classified in \cite{Klingler} all affine structures on compact complex surfaces.

In general, none of the local commuting vector fields are globally defined, due to the monodromy of the connection,
which is the linear part of the monodromy of the affine structure. However, the case of finite monodromy
can be related to stronger structure. For instance, when the monodromy is trivial, then the commuting vector fields
globalize to define a flat pencil of foliations in the sense of \cite{Lins}. 
When the monodromy is finite, of order $d$ say, then we get parallelizable $d$-webs.

The goal of this work is to investigate a singular version of these structures. 
A {\bf meromorphic affine structure} on $S$ is
a meromorphic affine connection $\nabla$ on $S$ 
(i.e. with meromorphic Christoffel symbols in local trivializations) with identically vanishing torsion and curvature.
We will say that the meromorphic affine structure is {\bf logarithmic} (resp. {\bf regular-singular})
if the corresponding affine connection, viewed as a linear meromorphic connection
on the tangent bundle, is logarithmic (resp. regular-singular) in the sense of \cite{Deligne}.
The polar set defines a divisor $D$ on $S$ and our aim in this paper is to provide
a sharp description of the regular affine structure at the neighborhood of a generic point of (the support of) $D$.
In order to list our models, it is convenient to describe the closed $1$-forms $\omega_t$ corresponding to 
constant $1$-forms $dx+tdy$ on $X=\C^2$ rather that vector fields or connection. 

\begin{THM}
Let $S$ be a complex surface, and $\nabla$ be a logarithmic affine structure on $S$ with polar divisor $D$.
Then, at a generic point $p$ of $D$, there are local coordinates $(S,p)\to(\C^2, 0)$ such that
the affine structure belongs to one of the following models (with pole $\{y=0\}$):
\begin{enumerate}
\item $\omega_t= dx + ty^{\nu}dy$, $\nu\in \C^*$,
\item $\omega_t= dx + t(\frac{dy}{y^{n}}+\frac{dy}{y})$, $n\in \Z_{>1}$,
\item $\omega_t= dx -y^n\ln(y)dy+ ty^ndy$, $n\in \Z_{\ge0}$,
\item $\omega_t= \frac{dy}{y^n}+(c-x)\frac{dy}{y}-\ln(y)dx+ tdx$, $n\in\Z_{>1}$, $c=0,1$.
\end{enumerate}    
\end{THM}

Here, generic point means any point outside a discrete subset that includes at least singular points of $D$. 
Outside the divisor $D$, we get a regular affine structure whose local charts are obtained by 
integrating $\phi=(\int\omega_0,\int\omega_\infty)$. The local affine charts degenerate along $D$,
and may have monodromy around.

One of our motivations comes from the study of singular $d$-webs
(see  \cite{Cerveau}, \cite {Falla}, \cite{PereiraPirio} or \cite{Robert}).
In the regular setting, a {\bf $d$-web} $\W=\Fo_1\boxtimes\cdots\boxtimes\Fo_d$ is locally 
given by $d$ foliations $\Fo_i$ in general position (i.e. pairwise transversal). 
We will say that $\W$ is {\bf parallelizable}, if there exist
local coordinates $(x,y)$ in which $\Fo_i:dx+t_idy$ for $i=1,\ldots,d$.
When $d\ge3$, then these normalizing coordinates are unique up to affine transformation
and such a $d$-web therefore defines an affine structure.
In the singular setting, a $d$-web is defined by a multi-section of the projectivized tangent bundle
$\pi:\Proj(T_S)\to S$, i.e. a surface $\Sigma_\W\subset \Proj(T_S)$ without vertical component, 
that intersects a generic fiber at $d$ distinct points. The discriminant $\Delta_\W$ in $S$
is the locus of points $p\in S$ where $\pi^{-1}(p)$ intersect $\Sigma_\W$ at $<k$ points:
it is the singular locus of $\W$. 
The singular $d$-web is said parallelizable if the regular $k$-web induced on $S\setminus \Delta_\W$
is parallelizable. It turns out that, for $d\ge3$, a singular parallelizable $d$-web defines a meromorphic
affine structure with regular-singular poles along $\Delta_\W$ (and smooth outside). 
We say that $\W$ is logarithmic if the affine structure is so. 
In that direction, we can prove:

\begin{THM}
Let $\W$ be a singular parallelizable $d$-web on $S$, $d\geq3$, 
with logarithmic singular points. Then, at a generic point of the discriminant $\Delta$,
$\W$ is contained in one of the following pencils:
\begin{enumerate}
\item $\{(dx)^q+ty^p(dy)^q=0\}_t$, with $(p,q)$ relatively prime positive integers,
\item $\{y^p(dx)^q+t(dy)^q=0\}_t$, with $(p,q)$ relatively prime positive integers,
\item $\{y^{n+1}dx+t(1+y^n)dy=0\}_t$, with $n$ positive integer.
\end{enumerate}
\end{THM}

For instance, in the first case, for each $t$ we get a $q$-web, except for $t=0$ or $\infty$
where we get a single foliation $dx$ or $dy$; then $\W$ is the superposition of several 
of these webs. In the last item, the monodromy is trivial and the web splits as a union
of $d$ foliations belonging to the pencil.

In the second section, we introduce the notions of affine connections, Riccati foliations
and pencils of foliations and the corresponding notions of curvature and torsion.
The regular setting is already described in \cite{Lizarbe}, and we explain how to 
adapt to the singular setting. 
In particular, we establish one-to-one correspondence between meromorphic affine structures 
and singular parallelizable Riccati foliations. 
An important example of Riccati foliations is induced by a pencil of foliations on a surface that will allow us to understand other cases.

In the third section, we prove our first main result about local classification of affine structures with logarithmic pole. 

In the fourth section we give a brief exposition of $d$-webs on surfaces, $d\geq 3$, with constant 
cross-ratio to later establish a relationship between Riccati foliations and webs.
Then we deduce our second main result about the local normal forms of logarithmic hexagonal $d$-webs.


\section{Affine structure, affine connection and Riccati foliations}\label{S:AffineRiccatiConnections}

An affine structure on a smooth complex surface $S$ is a maximal atlas of charts
$(\phi_i:U_i\to\C^2)$ with transition charts $\phi_j\circ\phi_i^{-1}$ induced by global
affine transformations 
$$\Aff=\left\{F:\C^2\to\C^2\ ;\ Z\mapsto AZ+B,\ \ \ A\in\mathrm{GL}_2(\C),\ B\in\C^2\right\}.$$
If $\phi:U\to\C^2$ is in the atlas, then any $F\circ\phi:U\to\C^2$ is also in the atlas, 
and any chart on $U$ belonging to the atlas takes this form.
Any local chart can be continued analytically along any path.
Indeed, given a path $\gamma:[0,1]\to S$ starting from $p_0=\gamma(0)\in U$,
then we can cover $\gamma$ by open sets $U_0=U,U_1,\ldots,U_n$ such that
\begin{itemize}
\item when $t\in[0,1]$ increases from $0$ to $1$, then $\gamma(t)$ intersects successively 
the $U_i$'s with $i$ increasing from $0$ to $n$,
\item intersections $U_i\cap U_{i+1}$ are contractible,
\item there is a well-defined chart $\phi_i:U_i\to\C$ in the affine atlas.
\end{itemize}
Then we have $\phi_i=F_{i,i+1}\circ \phi_{i+1}$ on $U_i\cap U_{i+1}$ so that
$F_{i,i+1}\circ \phi_{i+1}$ provides an analytic extension of $\phi_i$ on $U_{i+1}$.
Starting from $\phi_0=\phi$ on $U$, we can extend it successively as follows
$$\underbrace{\phi_0}_{U_0}=\underbrace{F_{0,1}\circ \phi_1}_{U_1}
=\underbrace{F_{0,1}\circ F_{1,2}\circ \phi_2}_{U_2}=\cdots
=\underbrace{F_{0,1}\circ F_{1,2}\circ\cdots\circ F_{n-1,n}\circ \phi_n}_{U_n}=:\phi^\gamma.$$
We can check that the analytic continuation $\phi^\gamma$ of $\phi$ along $\gamma$ depends
only on the homotopy type of $\gamma$ with fixed boundary. When $\gamma$ is a loop
(and $U_0=U_n$, $\phi_0=\phi_n=\phi$), then we get 
$$\phi^\gamma=\underbrace{F_{0,1}\circ F_{1,2}\circ\cdots\circ F_{n-1,n}}_{F_\gamma}\circ \phi=F_\gamma\circ\phi$$
and this defines the monodromy representation 
$$\pi_1(S,p_0)\to\Aff\ ;\ \gamma\mapsto F_\gamma.$$
One can locally pull-back constant vector fields $\C\langle\partial_x,\partial_y\rangle$ by affine charts
and we get, locally on $S$, a two-dimensional vector space of commuting vector fields.
Mind that none of these vector field is globally defined, due to the monodromy of the structure:
the linear part of $F_\gamma$ acts on this vector space in the natural way. 
We can encode these collections of local vector fields into a flat structure on $T_S$,
i.e. a affine connection (see section \ref{ss:AffineConnections}).
By duality, we can also pull-back the constant differential  $1$-forms $\C\langle dx,dy\rangle$
and get, locally on $S$, a two-dimensional vector space of closed $1$-forms; 
this can also be encoded into a flat linear connection on the cotangent bundle $\Omega^1_S$.
We can recover the affine charts (and structure) from these latter data by locally straightening 
commuting vector fields to $(\partial_x,\partial_y)$, or closed $1$-forms to $(dx,dy)$
(in fact, coordinates of $\phi$ are given by integrating closed $1$-forms).
A last object related to the affine structure is the pull-back of foliations by parallel lines
which locally defines a $1$-parameter family of foliations on $S$, a so called Veronese web,
that is encoded by a Riccati foliation on the projectivized tangent or cotangent bundle.

\subsection{Affine connections}\label{ss:AffineConnections}

There are several notions of connections in the litterature. From the point of view of algebraic geometers,
an {\bf affine connection} is a linear connection on the tangent bundle, i.e. a $\C$-linear map
\begin{equation}\label{eq:LinearConnection}
\nabla\ : \ T_S\to T_S\otimes\Omega^1_S
\end{equation}
satisfying the Leibniz rule $\nabla(fY)=df\otimes Y+f\nabla Y$. The contraction with
a vector field $X$ provides a $\C$-linear map
\begin{equation}\label{eq:AffineConnection}
T_S\times T_S\to T_S\ ;\ (X,Y)\mapsto \nabla_XY:=i_X(\nabla Y)
\end{equation}
which is $\Ol_S$-linear with respect to $X$, and satisfies the Leibniz rule with respect to $Y$.
This is the point of view of differential geometers. Finally, the equation for horizontal sections
$\nabla Y=0$, viewed as a Pfaffian system on the total space $V$ of $T_S$, defines a distribution
of $2$-planes transversal to the fibration $V\to S$. This gives a way to lift vector fields on $S$
to vector fields on $V$: this is a connection in the sense of Ehresmann.

In local coordinates $(x,y):U\to\C^2$ on $S$, one can trivialize the tangent bundle $T_S\vert_U$
by choosing the basis $(\partial_x,\partial_y)$. Then the linear connection writes
\begin{equation}\label{eq:MatrixConnection}
\nabla=d+\theta,\ \ \ \text{where}\ \ \ 
\theta=\left(\begin{array}{cc}
\theta_{11}&\theta_{12}\\
\theta_{21}&\theta_{22}
\end{array}\right)
\end{equation}
with $\theta_{ij}$ holomorphic $1$-forms on $U$
and Christoffel symbols are given by their coefficients $\theta_{ij}=\Gamma^i_{1j}dx+\Gamma^i_{2j}dy$.
The Ehresmann distribution is given by the corresponding system
\begin{equation}\label{eq:LinearSystem}
\nabla Z=0\ \ \ \Leftrightarrow\ \ \ \left\{\begin{matrix}
dz_1+\theta_{11} z_1+\theta_{12} z_2=0\\ dz_2+\theta_{21} z_1+\theta_{22} z_2=0
\end{matrix}\right.
\end{equation}
where $Z=z_1\partial_x+z_2\partial_y$.

We define the {\bf curvature} of an affine connection as
\begin{equation}\label{eq:curvatureAffineCon}
R^\nabla_{X,Y}Z=\nabla_X\nabla_YZ-\nabla_Y\nabla_XZ-\nabla_{[X,Y]}Z
\end{equation}
where $[X,Y]$ denotes the Lie bracket between vector fields.
We say that $\nabla$ is flat if the curvature vanishes identically.
Equivalently, the linear connection satisfies $\nabla\cdot\nabla=0$,
which locally writes 
\begin{equation}\label{eq:FlatnessMatrixConnection}
d\theta+\theta\wedge \theta=0\ \ \ \Leftrightarrow\ \ \ 
\left\{\begin{matrix} d(\theta_{11}+\theta_{22})=0\\ 
d\theta_{12}+(\theta_{11}-\theta_{22})\wedge\theta_{12}=0\\
d(\theta_{11}-\theta_{22})+2\theta_{12}\wedge\theta_{21}=0\\
d\theta_{21}+\theta_{21}\wedge(\theta_{11}-\theta_{22})=0
\end{matrix}\right.
\end{equation}
The flatness condition is equivalent to Frobenius integrability for the associate Ehresmann 
distribution, and therefore to the existence of a basis $(Y_1,Y_2)$ of local horizontal sections, 
i.e. $\nabla_X Y_i=0$ for all $X$.

We define the {\bf torsion} of an affine connection as
\begin{equation}\label{eq:torsionAffineCon}
T^\nabla(X,Y)=\nabla_XY-\nabla_YX-[X,Y].
\end{equation}
We say that $\nabla$ is torsion-free if the torsion vanishes identically.
In local coordinates, this is equivalent to
\begin{equation}\label{eq:TorsionMatrixConnection}
\theta_{11}(\partial_y)=\theta_{12}(\partial_x)\ \ \ \text{and}\ \ \ \theta_{21}(\partial_y)=\theta_{22}(\partial_x).
\end{equation}
A flat connection is torsion-free if, and only if, basis of horizontal sections define commuting vector fields
$[Y_1,Y_2]=0$. Indeed, we can always assume $Y_1=f(x,y)\partial_x$ and $Y_2=g(x,y)\partial_y$
in convenient coordinates so that the matrix connection writes 
$\theta=-\mathrm{diag}\left(\frac{df}{f},\frac{dg}{g}\right)$ and we have:
$$\text{torsion-free}\ \ \ \Leftrightarrow\ \ \ \frac{\partial f}{\partial_y}=\frac{\partial g}{\partial_x}=0
\ \ \ \Leftrightarrow\ \ \ [f(x,y)\partial_x,g(x,y)\partial_y]=0.$$

\subsection{Riccati foliation and Veronese web}\label{ss:RiccatiVeronese}

Here, we provide another point of view which is more convenient for our computations
(see  \cite{Lizarbe} for much more details).
In fact, we need a little bit less to encode an affine structure. It is well-known that 
projective linear transformations $\mathrm{PGL}(3,\C)$ acting on $\PP^2$, are
locally characterized by the fact that they send lines to lines. In particular, the affine
subgroup $\Aff$ stabilizing the line at infinity is characterized by the fact that it
preserves parallel lines. In other words, it preserves the $1$-parameter family 
of foliations $\Fo_t=\ker(dy-tdx)$ where $t\in\PP^1$ (setting $\Fo_\infty=\ker(dx)$).
When we pull-back on $S$ by affine charts, we get a Veronese web which is 
locally the data of a pencil of pairwise transversal foliations: 
\begin{equation}\label{eq:PencilFoliation}
\Fo_t=\ker(\omega_t),\ \ \ \omega_t=\omega_0-t\omega_\infty
\end{equation}
(see \cite{Lins} for instance). 
Globally on $S$, elements of this pencil are permuted by the monodromy of affine charts
and we fail to have a global pencil: we only have a so-called {\bf Veronese web}.

To properly define such a structure on $S$, we consider the projectivized tangent bundle
$\Proj(T_S)\to S$ which is a $\PP^1$-bundle over $S$ whose fiber at a point $p\in S$ 
is the set of directions through $p$. We will denote by $M$ the total space of $\Proj(T_S)$
and by $\pi:M\rightarrow S$ the natural projection. It is naturally equipped with a contact
structure (see below) so that the natural lifts of curves are Legendrian.
A foliation on $S$ corresponds naturally to a (smooth) section of this bundle
and a Veronese web corresponds to a flat structure on the bundle, i.e. a
$2$-dimensional foliation that is transversal to the fibration:
locally on $U\subset S$, $M\vert_U$ is foliated by sections, each of them producing
an element of the pencil. Below, we detail in local coordinates.

Starting from the affine structure, and the associate flat affine connection $\nabla$ on $T_S$,
we deduce the above flat structure by the action of $\nabla$ on directions (remind $\nabla$ is $\C$-linear).
In local coordinates $(x,y):U\to\C^2$, vectors $z_1\partial_x+z_2\partial_y$ in the fiber of $T_S$ are replaced by homogeneous 
coordinates $(z_1:z_2)=(1:z)$ with $z\in\PP^1$ in $\Proj(T_S)$. 
Therefore, the $\PP^1$-bundle writes
\begin{equation}\label{eq:LocalTrivialPTS}
\begin{matrix}
  \pi: M\vert_U=\Proj^1\times U&\rightarrow U \\
    ((1:z), (x,y)) &\mapsto (x,y),
\end{matrix}
\end{equation}
and the contact structure writes $dy = zdx$.
Then equations (\ref{eq:LinearSystem}) induce a Riccati type Pfaffian equation 
\begin{equation}\label{eq:RiccatiEquation}
\omega=dz+\alpha z^2+\beta z+\gamma=0,\ \ \ \alpha,\beta,\gamma\in\Omega^1(U)
\end{equation}
where
\begin{equation}\label{eq:FromConnectionToRiccati}
\left\{\begin{matrix}
\alpha=-\theta_{12}\hfill\\
\beta=\theta_{22}-\theta_{11}\\
\gamma=\theta_{21}\hfill
\end{matrix}\right.
\end{equation}
The Frobenius integrability condition of $\omega$ writes as follows
\begin{equation}\label{eq:FrobeniusRiccati}
\omega\wedge d\omega=0\ \Leftrightarrow\ 
\left\{\begin{matrix}
d\alpha+\alpha\wedge\beta=0,\\
d\beta+2\alpha\wedge\gamma=0,\\
d\gamma+\beta\wedge\gamma=0.
\end{matrix}\right.
\end{equation}
and is directly implied by flatness condition (\ref{eq:FlatnessMatrixConnection}).
We therefore get a foliation $\Ho^\nabla$ defining a flat structure on $\Proj(T_S)$.
Each leaf of $\Ho^\nabla$ defines a local section $(1:z)=(f(x,y):g(x,y))$
and therefore a foliation, generated by $f(x,y)\partial_x+g(x,y)\partial_y$.
More precisely, there exists a first integral for $\Ho^\nabla$ of the form
\begin{equation}\label{eq:FirstIntegRiccati}
F(x,y,z)=\frac{f_0(x,y)+zg_0(x,y)}{f_\infty(x,y)+zg_\infty(x,y)},\ \ \ 
\delta:=\det\begin{pmatrix}f_0&g_0\\ f_\infty&g_\infty\end{pmatrix}\not=0
\end{equation}
i.e. one retrieve the Riccati equation (\ref{eq:RiccatiEquation}) from the Pfaffian equation $dF=0$ yielding
\begin{equation}\label{eq:FirstIntegToRicc}
\left\{\begin{matrix}
\alpha=\frac{g_0dg_\infty-g_\infty dg_0}{\delta}\hfill\\
\beta=\frac{f_0dg_\infty-g_\infty df_0+g_0df_\infty-f_\infty dg_0}{\delta}\\
\gamma=\frac{f_0df_\infty-f_\infty df_0}{\delta}\hfill
\end{matrix}\right.
\end{equation}
We note that $F$ can be deduced from a basis of $\nabla$-horizontal sections
$$M=\begin{pmatrix}f_0&f_\infty\\ g_0&g_\infty\end{pmatrix},\ \ \ dM+\theta M=0.$$
The leaves of the foliation $\Ho^\nabla$ are defined by fibers $F(x,y,z)=t$, $t\in\PP^1$,
and, taking into account the contact structure $dy-zdx=0$, we retrieve the local pencil 
of foliations defining the Veronese web structure:
$$F(x,y,z)=t$$
\begin{equation}\label{eq:FirstIntegToPencil}
\Updownarrow
\end{equation}
$$(\underbrace{f_0(x,y)dx+g_0(x,y)dy}_{\omega_0})-t(\underbrace{f_\infty(x,y)dx+g_\infty(x,y)dy}_{\omega_\infty})=0.$$
More generally, a flat structure on $\Proj(T_S)\to S$ is a Riccati foliation $\Ho$, 
locally defined in trivialization chart
by a Riccati equation (\ref{eq:RiccatiEquation}) satisfying flatness condition
(\ref{eq:FrobeniusRiccati}). This is the way to define a Veronese web on $S$.

To define an affine structure, it remains to characterize those Veronese webs
that are locally parallelisable, i.e. equivalent to the pencil of constant foliations 
up to change of coordinates (flat pencils in \cite{Lins}). 
In other words, it remains to translate the torsion free condition on the Riccati equation. 
For this, we note that we can lift a Riccati equation (\ref{eq:RiccatiEquation})
to a unique torsion free connection (\ref{eq:MatrixConnection}) by setting
(\ref{eq:FromConnectionToRiccati}) and 
\begin{equation}\label{eq:LiftRicToCon}
\kappa=\theta_{11}+\theta_{22},\ \ \ \text{with}\ \ \ 
\kappa=(2\gamma_2-\beta_1)dx+(\beta_2-2\alpha_1)dy
\end{equation}
under notation
\begin{equation}\label{eq:alphabetagammai}
\left(\begin{array}{c}
\alpha\\
\beta\\
\gamma
\end{array}\right)
=\left(\begin{array}{c}
\alpha_1\\
\beta_1\\
\gamma_1
\end{array}\right)dx+
\left(\begin{array}{c}
\alpha_2\\
\beta_2\\
\gamma_2
\end{array}\right)dy,
\end{equation}
Then, the resulting connection is flat if, and only if $d\kappa=0$. From \cite[Proposition 3.2]{Lizarbe},
the $2$-form $d\kappa$ does not depend on the choice of local coordinates $(x,y)$
and defines a global $2$-form on $S$: we call {\bf torsion} and denote $K(\Ho)\in\Gamma(S,\Omega^2_S)$ 
this global $2$-form. We will say that the Veronese web (or Riccati foliation) is torsion-free if $d\kappa=0$.

Following \cite{Lins}, a pencil of foliations $\Fo_t$ is defined by a family of Pfaffian equations
\begin{equation}\label{eq:Pencil}
(\omega_0-t\omega_\infty=0)_t\ \ \ \text{with}\ \ \ \omega_0\wedge\omega_\infty\not=0
\end{equation}
where the vector space $\C\langle\omega_0,\omega_\infty\rangle$
is well-defined up to multiplication by a function. In particular, one can assume 
$\omega_0$ or $\omega_\infty$ to be a closed $1$-form, but not both of them in general.
Then are equivalent:
\begin{itemize}
\item the pencil $\Fo_t$ is flat in the sense of \cite{Lins},
\item one can define $\Fo_t$ by (\ref{eq:Pencil}) with $d\omega_0=d\omega_\infty=0$,
\item $\Fo_t$ is locally equivalent to the linear pencil $dy-tdx=0$,
\item the corresponding Riccati foliation on $\Proj(T_S)$ is torsion-free: $d\kappa=0$.
\end{itemize}

\subsection{Meromorphic affine structures and singular Riccati foliations}\label{Sec:MeroAffine}

Let $S$ be a smooth complex surface, $M=\Proj(T_S)$ be the projectivization of the tangent 
bundle of $S$, and $\pi:M\rightarrow S$ the natural projection. A singular Riccati foliation $\Ho$ 
on $M$ is a singular holomorphic foliation of codimension 1 on $M$, transverse to the 
generic fiber of $\pi$.

The (effective) {\it polar divisor} $D$ of the Riccati foliation is defined as the direct
image under $\pi$ of the tangency divisor between $\Ho$ and the vertical foliation
defined by the fibers of $\pi$. 

Consider a trivialization of $M$ over an open set $U$ like in (\ref{eq:LocalTrivialPTS})
with coordinates $(x,y):U\to\C^2$ and contact structure given by $dy = zdx$. 
Then the foliation $\Ho$ is given by a non zero 
meromorphic 1-form $\omega$ of the type:
\begin{equation}\label{eq:Ricc1Form}
\omega=dz+\alpha z^2+\beta z+\gamma,\ \ \ \alpha,\beta,\gamma\in\Gamma(U,\Omega^1(D))
\end{equation}
where $\alpha$, $\beta$ and $\gamma$ are meromorphic 1-forms on $U$, 
and satisfying moreover
Frobenius integrability condition (\ref{eq:FrobeniusRiccati}).
The poles of $\alpha$, $\beta$ and $\gamma$ define the polar divisor $D$ of $\Ho$.

The torsion of a singular Riccati foliation $\Ho$ on $M=\Proj(T_S)$ is a meromorphic $2$-form
$K(\Ho)\in  \Gamma(S,\Omega_S^2(2D))$ which is locally defined by $d\kappa$
where $\kappa$ is given by (\ref{eq:LiftRicToCon}). We then say that $\Ho$ is torsion-free
when $K(\Ho)$ is identically vanishing.
A {\bf meromorphic affine structure} is the data of a torsion-free singular Riccati foliation on $\Proj(T_S)$.
From section \ref{ss:RiccatiVeronese}, we see that we get an affine structure in the usual sense
on the complement $S\setminus D$ of the polar locus. But the affine charts degenerate along $D$.

We will say that the affine structure is {\bf logarithmic} (resp.  {\bf regular-singular})
if it is locally induced by a logarithmic (resp.  regular-singular) connexion 
$$\nabla_i\ :\ T_S\vert_{U_i}\to T_S\vert_{U_i}\otimes\Omega^1_{U_i}(D)$$
in the sense of Deligne \cite{Deligne} (see also \cite{Novikov}). 
These properties
can be checked at a generic point of all irreducible components of $D$
and translate as follows on the Riccati foliation.

\begin{prop}\label{prop:RiccLogReg}Let $\Ho$ be a torsion-free singular Riccati foliation 
on $\Proj(T_S)$. Then are equivalent:
\begin{itemize}
\item the affine structure is logarithmic,
\item $\nabla_\Ho$ is logarithmic (see definition below),
\item $\omega$ and $d\omega$ have at most simple poles, 
where $\omega$ is the Riccati $1$-form defining $\Ho$ in local charts.
\end{itemize}
Are also equivalent:
\begin{itemize}
\item the affine structure is regular-singular,
\item there is a bimeromorphic bundle transformation $\psi:\Proj(T_S)\dashrightarrow P$
such that $\psi_*\Ho$ is a logarithmic Riccati foliation on $P$.
\item in local trivializations of $\Proj(T_S)$, solutions of the Riccati equation $\omega$
have polynomial growth when approaching $D$.
\end{itemize}
\end{prop}

There is a unique flat and torsion-free meromorphic connection 
$$\nabla_\Ho\ :\ T_S\to T_S\otimes\Omega^1_S(D)$$
lifting the flat structure $\Ho$ on $T_S$: it is defined in charts by identities
(\ref{eq:FromConnectionToRiccati}) and (\ref{eq:LiftRicToCon}). 
We do not know whether $\Ho$ regular-singular $\Rightarrow$ $\nabla_\Ho$ is regular-singular.

\begin{proof} Due to (\ref{eq:FromConnectionToRiccati}) and (\ref{eq:LiftRicToCon}),
we see that coefficients of $\omega$ have simple poles if, and only if, coefficients 
of the matrix connection (\ref{eq:MatrixConnection}) has simple poles. Moreover, 
the same formula together with $d\kappa=0$ shows that the same equivalence 
holds true for differential of coefficients, i.e. $d\theta$ has simple poles if, and only if,
$d\omega$ has simple poles. 

For the second part, we know (see \cite{Deligne}, or \cite{Novikov}) that $\nabla_\Ho$ 
is regular-singular singular points of linear connections are characterized by solutions 
having polynomial growth along the divisor $D$ (which needs not be reduced in that case),
and so must be those solutions of $\Ho$.
This growth property is invariant under bimeromorphic bundle transformation.
We can therefore assume that $\Ho$ is in normal form like in \cite[Theorem 1]{LP},
minimizing the pole order.
Then, it has simple poles otherwise it would have irregular singular points with exponential growth,
and poles can only be logarithmic since they cannot be erased by bimeromorphic bundle transformation.
Of course, logarithmic implies polynomial growth for Riccati equation (same computation as for linear connection).
Conversely, if solutions of $\Ho$ have polynomial growth then so are the solutions 
of the unique trace-free connection $\nabla^0$ lifting $\Ho$ in a local trivialization of $\Proj(T_S)$
(by setting $\theta_1+\theta_2=0$ instead of (\ref{eq:LiftRicToCon})).
\end{proof}

As a particular case of affine structures, we have those defined by a global pencil of foliations.
In the singular case, a singular pencil of foliation is defined by a pencil
$\omega_t=\omega_0-t\omega_\infty$ of global meromorphic $1$-forms, 
such that $\omega_0\wedge\omega_\infty\not\equiv0$.
In that direction, we have

\begin{prop}The following data are equivalent:
\begin{itemize}
\item a flat singular pencil on $S$, 
\item a torsion-free Riccati foliation on $\Proj(T_S)$ with regular singula\-ri\-ties and trivial monodromy.
\end{itemize}
\end{prop}

\begin{proof}
A singular pencil, with singular locus $D$ given by poles and zeroes of $\omega_0\wedge\omega_\infty$, 
clearly implies the existence of a Riccati foliation on $\Proj(T_U)$ where $U=S\setminus D$. 
By construction, the Riccati foliation admits a meromorphic first integral given by (\ref{eq:FirstIntegRiccati})
and therefore extends with regular-singular points. Conversely, a Riccati foliation with trivial monodromy 
defines a pencil of foliations outside the polar locus; if it is regular-singular, then these foliations (horizontal sections)
extend meromorphically on $S$, to form a global pencil. 
\end{proof}

We will study Riccati foliations at the neighborhood of a point $p\in D$ such that $D$ is smooth at $p$
and logarithmic (or regular-singular).
We will use the following classical result (see \cite[Proposition 1.1.16]{Cousin} or \cite[Theorem 1.6]{Novikov}):

\begin{prop}\label{P:gaugetransformation} 
Let $D=\{y = 0\}$ be a smooth divisor of $(\C^2, 0)$ and $\Ho$
be a Riccati foliation on $\Proj^1 \times (\C^2, 0)\rightarrow (\C^2, 0)$, having
polar divisor $D$. If $\Ho$ is logarithmic (resp. regular-singular) along $D$,
then, up to biholomorphic (resp. bimeromorphic) bundle transformation,
we can assume that the foliation $\Ho$ is given 
by one of the following equations:
\begin{enumerate}
\item $dz = \lambda z\frac{dy}{y}$, $\lambda\in \C$, or
\item $dz = (nz + y^n)\frac{dy}{y}$, $n\in \Z_{\geq 0}$. 
\end{enumerate}
\end{prop}

\section{Local classification of torsion-free Riccati foliations with logarithmic pole}

Let us start with Lemmae about classification of local multivalued functions and $1$-forms
in dimension one. The first one is very classical.

\begin{lemma}\label{L:multivaluedfunction} 
Let $f$ be a (possibly multivalued) function on $(\C,0)$ of the form $f(x)=x^\nu h(x)$ where $\nu\in\C^*$
and $h(x)$ holomorphic and non vanishing at $0$.
Then $f$ is conjugated to $x^{\nu}$, i.e.  $f(x)=\varphi^*x^\nu:=\left(\varphi(x)\right)^\nu$ for some local diffeomorphism $\varphi$.
\end{lemma}

\begin{proof}
If we write $\varphi(x)=x u(x)$ with $u(x)$ holomorphic and non vanishing on $(\C,0)$, then 
conjugacy $f=\varphi^*x^\nu$ is reduces to $h=u^\nu$. Since $h$ is non vanishing, 
there is a local holomorphic determination of $\log(h)$, and a solution is given by 
$u(x)=\exp(\nu^{-1}\log(h(x)))$.
\end{proof}

For multivalued 1-forms, we have:

\begin{lemma}\label{L:multivaluedform} 
Let $\omega=x^\nu u(x)dx$ be a multivalued $1$-form on $(\C,0)$ where
$\nu\in\C$ and $u(x)$ holomorphic and non vanishing at $0$.
\begin{itemize}
\item If $\nu\notin \Z_{<0}$, then $\omega$ is conjugated to $x^\nu dx$, i.e. 
$\omega=\varphi^*(x^\nu dx)$ for some local diffeomorphism $\varphi(x)$.
\item If $\nu\in \Z_{<0}$, then $\omega$ is meromorphic. Denote by 
$n=-\nu\in\Z_{>0}$ its pole order, and by $\lambda\in\C$ its residue at $0$. Then
\begin{enumerate} 
\item If $n=1$, then $\omega$ is conjugated to $\lambda \frac{dx}{x}$.   
\item If $n>1$, then $\omega$ is conjugated to 
$\frac{dx}{x^n}+\lambda\frac{dx}{x}$. 
\end{enumerate}
\end{itemize} 
\end{lemma}

\begin{proof}Assume first $\nu\notin \Z_{<0}$. Then $\nu+1\not=0$, and we can rescale $x$
by an homothecy to set $u(0)=1$. Write $\varphi(x)=x\exp(g(x))$ and substitute in $\omega=\varphi^*(x^\nu dx)$.
Then we find:
$$u(x)=e^{(\nu+1)g(x)}(1+x g'(x)),$$
which means that $y=g(x)$ is solution of the differential equation
$$x\frac{dy}{dx}=e^{-(\nu+1)y}u(x)-1.$$
The right-hand-side expands as:
$$x\frac{dy}{dx}=u_1 x-(\nu+1)y+\text{h.o.t.}$$
By Briot-Bouquet, there exists a holomorphic solution $y=g(x)$ with initial condition $g(0)=0$ 
provided that $\nu+1\not\in\Z_{\le0}$, which proves the first part of the statement.

Assume now that $\omega$ is meromorphic (but non holomorphic). 
If $n=1$, then write $\omega=\lambda\frac{dx}{x}\left(1+h(x)\right)$ with $h(x)$ holomorphic 
and vanishing at $0$. Then write $\varphi(x)=x u(x)$ with $u(x)$ holomorphic and non vanishing at $0$.
Then $\omega=\varphi^*(\lambda\frac{dx}{x})$ is equivalent to $\frac{du}{u}=h\frac{dx}{x}$
which can be solved by setting $u(x):=\exp\int h(x)\frac{dx}{x}$.

If $n>1$, then taking appart the residual part, we can write 
$$\omega=-d\left(\frac{h(x)}{(n-1)x^{n-1}}\right)+\lambda\frac{dx}{x}$$
with $h(x)$ holomorphic and non vanishing. Moreover, up to homothecy, we can assume
$h(0)=1$. Then write $\varphi(x)=x u(x)$, and conjugacy equation yields 
$$-d\left(\frac{h(x)}{(n-1)x^{n-1}}\right)=\frac{d\varphi}{\varphi^n}+\lambda\frac{du}{u}$$
(we have simplified residues). After integration, we get the functional equation:
$$h=\frac{1}{u^{n-1}}-(n-1)\lambda x^{n-1}\log(u).$$
Considering the main determination of $\log(u)$, the right-hand-side is holomorphic and vanishes
at $(x,u)=(0,1)$, and has non zero derivative along $\partial_u$, so that Implicit Function Theorem
provides a holomorphic solution $u(x)$ with $u(0)=1$.
\end{proof}

Let us now consider a logarithmic affine structure on a surface $S$, and let $D$ be the (reduced) polar
divisor. The affine structure is defined by a torsion-free Riccati foliation on $\Proj(T_S)$, and at the
neighborhood of any point $p\in S\setminus D$, by a flat pencil of foliations defined by 
$$\omega_t=\omega_0+t\omega_\infty\ \ \ \text{where}\ \ \ d\omega_0= d\omega_\infty=0,\ \ \ \text{and}\ \omega_0\wedge \omega_\infty\not=0.$$
Moreover, the vector space $V:=\C\langle\omega_0,\omega_\infty\rangle$ does not depend on any choice
and the analytic continuation of $\omega_0,\omega_\infty$ around branches of the divisor $D$
gives rise to a representation 
$$\pi_1(S\setminus D,p)\to \mathrm{GL}(V),$$
the linear monodromy of the affine structure. In particular, the $1$-forms $\omega_0,\omega_\infty$
do not extend, even meromorphically, along $D$, but define multivalued $1$-forms around.
Consider an irreducible component $D_0$ of $D$. Then, the conjugacy class of the local monodromy
around $D_0$ is well-defined, and might not fix any non zero $1$-form. However, in the logarithmic
case, some foliations of the pencil must extend as explained below:

\begin{prop} \label{P:multiformclosed} 
Let $\Ho$ be a torsion-free Riccati foliation on $\Proj(T_S)$ and $p\in D$
be a point on the smooth part of (the support of) $D$. Assume that the structure is regular-singular 
at $p$. Then the projective monodromy has at least one fixed point, and each fixed point gives rise 
to a non trivial $1$-form $\omega\in V$ with multiplicative monodromy around $D$, and the corresponding foliation
$\Fo_\omega$ extends as a (possibly singular) holomorphic foliation on $(S,p)$. 

Moreover, if $\Fo_\omega$ is transversal to $D$ at $p$, then $\omega$ extends as 
a closed holomorphic (and non vanishing) $1$-form at the neighborhood of $p$.
\end{prop}

\begin{proof} The monodromy of the affine structure is the projectivization of the linear monodromy
of acting on $V$ as above. The projective monodromy around $D$ at $p$ is a Moebius transformation
which has one or two fixed points, or is the identity. Each fixed point corresponds to a foliation $\Fo$
of the pencil which is uniform around $D$. It also correspond to a local section of $\Proj(T_S)$
outside of $D$ that is invariant by the Riccati foliation (i.e. a leaf). By Fuchs Theory (see also \cite{Deligne}),
because the Riccati equation is regular-singular, this section extends meromorphically along $D$,
therefore defining a singular foliation extending $\Fo$ along $D$. 

Finally, if $\Fo$ is transversal to $D$, we can write locally $\Fo=\ker(dy)$ and $D=[x=0]$
in local coordinates $(x,y)$ at $p$.
Therefore, we can write $\omega=f(x,y)dy$ for a multivalued function $f$. But $\omega$ is closed,
so $f(x,y)=f(y)$ with $f$ holomorphic near $y=0$; we promptly deduce that $\omega$ extends at the origin.
The function $f$ does not vanish at $0$, otherwise it would not define a regular pencil along $y=0$,
contradicting that the structure is regular outside $x=0$.
\end{proof}

We are now going to solve the local classification problem at a generic point of a branch $D_0$ of $D$
where the local projective monodromy has two fixed points.

\begin{thm} \label{T:affinestructure}
Consider a logarithmic affine structure on $S$ with polar divisor $D$
and let $D_0$ be a branch of $D$ with semi-simple local monodromy.
Then, at a generic point of $D_0$, the affine structure is described 
by one of the model below:
\begin{itemize}
\item $\omega_t= dx + ty^{\nu}dy$, $\nu\in \C^*$,
\item or $\omega_t= dx + t(\frac{dy}{y^{n}}+\frac{dy}{y})$, $n\in \Z_{>1}$.  
\end{itemize} 
\end{thm}

In the statement, by generic point, we mean outside of a discrete set of points along $D$.

\begin{proof}  By assumption, the projective monodromy around $D_0$ has at least 2 fixed points.
Applying Proposition \ref{P:multiformclosed} to these two fixed points, we get that 
the affine structure is defined around any point $p\in D_0$ which is smooth for $D$
by the pencil $\omega_t=\omega_0+t\omega_\infty$ where $\omega_0,\omega_\infty$ are
multivalued closed $1$-forms with multiplicative monodromy $\omega_i\mapsto c_i\omega_i$
around $D_0$, $c_i\in\C^*$, $i=0,\infty$. Moreover, the corresponding foliations $\Fo_0,\Fo_\infty$
extend as (possibly singular) holomorphic foliations along at $(S,p)$. Moreover, the linear monodromy 
around $D_0$ is given in the basis $(\omega_0,\omega_\infty)$ by
$\begin{pmatrix}c_0&0\\ 0&c_\infty\end{pmatrix}$; if $c_0\not=c_\infty$ then 
no other foliation $\Fo_t:\{\omega_t=0\}$ of the pencil is preserved by the monodromy.
At a generic point $p\in D_0$, these two foliations are smooth (non singular) and 
each of them is either transversal to $D$, or $D$ is a local leaf of the foliation.
In other words, we exclude special points where foliations are singular, or have
isolated tangencies with $D$, or extra tangency between them. 
Then, we have $4$ possibilities for $(\Fo_0,\Fo_\infty)$ up to permutation:
\begin{itemize}
\item $\Fo_0$ and $\Fo_\infty$ are transversal, and both transversal to $D_0$.
\item $\Fo_0$ and $\Fo_\infty$ are transversal, and $D_0$ is a leaf of $\Fo_0$.
\item $D_0$ is a common leaf of $\Fo_0$ and $\Fo_\infty$, and these foliations are transversal outside of $D_0$.
\item $\Fo_0$ and $\Fo_\infty$ are tangent along $D_0$, but are transversal to $D_0$, 
and are transversal outside of $D_0$.
\end{itemize}
The proof ends by a studying each of these $4$ cases separately.

\begin{lemma} \label{L:Pencil2} 
If $\Fo_0$ and $\Fo_\infty$ are transversal, then the polar divisor $D$ of the structure 
must be invariant by either $\Fo_0$, or $\Fo_\infty$.
\end{lemma}

\begin{proof} Assume by contradiction that $D$ is not invariant by either $\Fo_0$ or $\Fo_\infty$.
At a generic point of $D$, the divisor is transversal to $\Fo_0$ and $\Fo_\infty$.
We apply the last part of Proposition \ref{P:multiformclosed} and get that $\Fo_0$ and $\Fo_\infty$
are defined by holomorphic closed $1$-forms $\omega_0,\omega_\infty$ at $p$,
which write $\omega_0=dx$ and $\omega_\infty=dy$ in convenient local coordinates at $p$.
But clearly, the pencil is regular at $p$, contradiction.
\end{proof}

\begin{lemma}\label{L:Pencil1} 
If  $\Fo_0$ and $\Fo_\infty$ are transversal and $D$ is invariant by $\Fo_\infty$, 
then there exist local coordinates in which the structure is generated by 
\begin{enumerate}
\item $\omega_t= dx + ty^{\nu}dy$, $\nu\in \C^*$,
\item or $\omega_t= dx + t(\frac{dy}{y^{n}}+\frac{dy}{y})$, $n\in \Z_{>1}$.  
\end{enumerate} 
\end{lemma}

\begin{proof} 
We can choose local coordinates such that $\Fo_0=\ker(dx)$ and $\Fo_\infty=\ker(dy)$.
Moreover, by Proposition \ref{P:multiformclosed}, we can write $\omega_0=f(x)dx$
with $f$ holomorphic and non vanishing; after coordinate change $x:=\varphi(x)=\int f(x)dx$,
we can assume $\omega_0=dx$ and $\omega_\infty=g(x,y)dy$ for a multivalued function $g$. 
Since $\omega_\infty$ is closed, we get $g(x,y)=g(y)$, still multivalued.
The corresponding Riccati equation in the local trivialization of $\Proj(T_S)$ is deduced as follows:
$$
\omega_t=dx+t g(y)dy \ \ \ \Rightarrow\ \ \ \frac{1}{t}=g(y)\frac{dy}{dx}=:g(y)z
$$
which, after derivation gives
$$0=zdg+gdz \ \ \ \leadsto\ \ \ \omega=dz+\frac{dg}{g}z=0.$$
In case $\omega$ has a multiple pole, then we easily deduce that we have an irregular singular point.
So, in the regular-singular case, we must have a simple pole and we again easily check that it is logarithmic.
Applying Lemma \ref{L:multivaluedform}, we deduce that, after $y$-coordinate change, 
we get one of the model of the statement. Indeed, in case $\nu\in\Z_{<0}$ in Lemma \ref{L:multivaluedform}, 
we can replace $\omega_\infty$ by a multiple to normalize the residue $\lambda=1$.
\end{proof}

\begin{lemma} \label{L:Pencil3} 
If $\Fo_0$ and $\Fo_\infty$ are tangents along $D$, and transversal to $D$,
then there exist local coordinates in which the structure is generated by  
$$\omega_t=dx + td(y^{n}),\ \ \ n\in\Z_{\ge2}.$$
In the regular singular case, we also get the following local models
$$\omega_t=f(x)dx + td(x+y^{n}),\ \ \ n\in\Z_{\ge2},\ f\in\Ol^*.$$
\end{lemma}

\begin{proof} 
We first apply \cite[Lemme 5.3]{LorayS} to deduce that there is a change of coordinates 
such that $\Fo_0=\{dx=0\}$, and $\Fo_\infty=\{d(x+y^{n})=0\}$, $n\geq2$, so that $D=\{y=0\}$, 
$\omega_0=f(x,y)dx$ and $\omega_{\infty}=g(x,y)d(x+y^{n})$. 
Since $\omega_0$ and $\omega_\infty$ are closed, and $\omega_0$ extends holomorphically at $p$
(Proposition \ref{P:multiformclosed}), we have $f(x,y)=f(x)$ and $g(x,y)=g(x+y^{n})$
with $f,g\in\Ol^*$. In fact, we can normalize the restriction $\omega_\infty\vert_D=g(x)dx$
to $dx$ by a change of $x$-coordinate and then reapply \cite[Lemme 5.3]{LorayS} to get 
the normal form
$$\omega_0=f(x)dx\ \ \ \text{and}\ \ \ \omega_{\infty}=d(x+y^{n}).$$
This pencil induces the Riccati equation
\[
\omega=dz-\frac{1}{ny^{n-1}}\frac{df}{f}+\left((n-1)\frac{dy}{y}-\frac{df}{f}\right)z.
\]
One can check that, after meromorphic gauge transformation $\tilde z=zy^{n-1}$, the Riccati foliation becomes holomorphic:
the above local model is always regular-singular. However, after derivation, we get that $d\omega$ has a pole of order 
$n\ge2$ unless $f'(x)=0$, so $f$ is a constant: we obtain the pencil of the statement (with a different basis).
\end{proof}

\begin{lemma} \label{L:Pencil4} 
If $D$ is invariant by $\Fo_0$ and $\Fo_\infty$, then $\Fo_0$ and $\Fo_\infty$
are generic elements of the pencils described in Lemma \ref{L:Pencil1}
(with $\nu\in\Z_{<0}$ in the first case).
\end{lemma}

We start recalling a local technical Lemma for foliations tangent along a common leaf (\cite[Lemma 5]{LorayV}, \cite[Proposition 1]{Thom}, \cite[Lemma 4.7, 4.9]{LTT}):

\begin{prop}\label{Prop:Tang}
Let $\Fo$ and $\Go$ be two smooth foliations at $(\C^2,0)$
that have $y=0$ as a common leaf, and transversal outside. Then, up to change of coordinates,
we can assume $\Fo$ and $\Go$ are defined by the respective first integrals
$$f(x,y)=y\ \ \ \text{and}\ \ \  g(x,y)=y+xy^{k+1},\ \ \ k\in\Z_{\ge0},$$
where $k+1$ is the order of tangency, defined by the vanishing order of $df\wedge dg$ along $y=0$.

Moreover, this normalisation is not unique: if a diffeomorphism $\Phi(x,y)$ of $(\C^2,0)$
commutes with the normal form, then we obtain new first integrals that factor through the initial ones as follows:
$$f\circ\Phi=\varphi\circ f\ \ \ \text{and}\ \ \ g\circ\Phi=\psi\circ g$$
for one-dimensional diffeomorphisms $\varphi$ and $\psi$. Then this provides a one-to-one
correspondance between 
\begin{itemize}
\item symetries $\Phi$ of the normal form,
\item pairs $(\varphi,\psi)$ of diffeomorphisms coinciding up to order $k+1$, 
i.e. satisfying $\varphi(y)-\psi(y)=o(y^{k+1})$.
\end{itemize}
\end{prop}

\begin{proof}
The first part of the statement is proved in each of the quoted references.
The second part is not stated like this, and we give the proof. 
We just have to explain
how to reconstruct $\Phi$ from the pair $(\varphi,\psi)$. Clearly, by action on $f$, we see
that $\Phi=(\phi(x,y),\varphi(y))$. Then, action on $g$ gives
$$\psi(y+xy^{k+1})=\varphi(y)+\phi(x,y)(\varphi(y))^{k+1}$$
which already gives the condition $\psi(y)=\varphi(y)$ mod $y^{k+2}$. But this rewrites
$$\phi(x,y)=\frac{\psi(y+xy^{k+1})-\varphi(y)}{y^{k+1}}=\frac{\psi(y)-\varphi(y)}{y^{k+1}}+\psi'(0)x+o(x)$$
which shows that $\phi(x,y)$ is uniquely determined, holomorphic, of the form $\phi(x,y)=\psi'(0)x+cy+\text{h.o.t.}$.
\end{proof}

\begin{proof}[Proof of Lemma \ref{L:Pencil4}]
Following Proposition \ref{Prop:Tang},
there is a change of coordinates such that  
$$\Fo_0=\ker(dy),\ \ \ \text{and}\ \ \ \Fo_\infty=\ker(d(y+xy^{n+1})),\ \ \ n\in\Z_{\geq 0},$$ 
so that $D=\{y=0\}$. The pencil writes
$$\omega_0=g(x,y)dy\ \ \ \text{and}\ \ \ \omega_{\infty}=f(x,y)d(y+xy^{n+1})$$
and by closedness of the $1$-forms, we can furthermore assume
$$f(x,y)=f_0(y+xy^{n+1}),\ \ \ \text{and}\ \ \ g(x,y)=g_0(y)$$
for possibly multivalued holomorphic functions $f_0,g_0$ on the punctured neighborhood of $0\in\mathbb C$.
Hence, this pencil induces a Riccati foliation given by
\begin{equation}\label{Eq:RiccTang}
\omega=dz-\left(\frac{df}{f}-\frac{dg}{g}+(n+1)\frac{dy}{y}\right)z
\end{equation}
$$-\left(\frac{(1+(n+1)xy^{n})}{y^{n+1}}\left(\frac{df}{f}-\frac{dg}{g}\right)
+n(n+1)x\frac{dy}{y^2}+(n+1)\frac{dx}{y}\right)z^2.$$
By considering the term in $z$, we see that $\frac{df}{f}-\frac{dg}{g}$ must be a meromorphic $1$-form.
After gauge transformation $\tilde z=\frac{z}{y^{n+1}}$, we get 
$$dz-\left(\frac{df}{f}-\frac{dg}{g}\right)z$$
$$-\left((1+(n+1)xy^{n})\left(\frac{df}{f}-\frac{dg}{g}\right)+n(n+1)xy^{n-1}dy+(n+1)y^n dx\right)z^2.$$
This latter equation has an irregular singular point whenever
$\frac{df}{f}-\frac{dg}{g}$ has a multiple pole; we deduce that  $\frac{df}{f}-\frac{dg}{g}$
must have at most simple pole, and in that case, the Riccati foliation is regular-singular. 

Assume now that the structure is logarithmic. Considering the coefficient of $z^2$ in (\ref{Eq:RiccTang}),
and its restriction to $x=0$, we see that 
\begin{itemize}
\item $\frac{df}{f}$ must be logarithmic (coefficient of $dx$)
\item $\frac{df}{f}-\frac{dg}{g}$ must be holomorphic, vanishing at order $n$ along $y=0$  (coefficient of $dy$)
\end{itemize}
Applying Lemma \ref{L:multivaluedform} to $f_0(y)dy$, we have two cases:
$$f_0(y)dy=\varphi^* y^\nu dy,\ \ \ \nu\in\C^*,\ \ \ \text{or}\ \ \ \varphi^* \frac{dy}{y^{k+1}}+\frac{dy}{y}$$
for some local diffeomorphism $\varphi(y)$.

First case: $f_0(y)dy=\varphi^* y^\nu dy$. Once more we apply Proposition \ref{Prop:Tang}
and we get that there exists a change of coordinates on $(\C^2,0)$ such 
that $f(x,y)=(y+xy^{n+1})^{\nu}$. Therefore, we have $\frac{dg}{g}=\nu\frac{dy}{y}+h(y)dy$ with $h(y)$
holomorphic, vanishing at order $n$ at $y=0$. Substituting in (\ref{Eq:RiccTang}),
and assuming first $n>0$, we see that the non logarithmic terms are:
$$(\nu+n+1)\left(\frac{dx}{y}+nx\frac{dy}{y^2}\right)$$
and we deduce that $\nu=-n-1$. In the case $n=0$, the non logarithmic term is:
$$(\nu+1)\frac{dx}{y}$$
and we find $\nu=-1$ as before.

Second case: a similar study shows that, when $f_0(y)dy=\varphi^*\frac{dy}{y^{k+1}}+\frac{dy}{y}$,
again, $k=n$, and we can finally assume $f_0(y)=\frac{1}{y^{n+1}}+\frac{1}{y}$ or $\frac{1}{y^{n+1}}$. 
Taking into account the vanishing order of $\frac{df}{f}-\frac{dg}{g}$, we deduce that $g_0(y)-f_0(y)$
is holomorphic, and $g_0(y)$ is conjugated to $f_0$ by a diffeomorphism tangent to the identity
up to order $n+1$. Therefore, Proposition \ref{Prop:Tang}, we can assume moreover that $g_0(y)=f_0(y)$.
We then arrive to the normal form 
\[
 \omega_t=\frac{dy}{y^{n+1}}+\epsilon\frac{dy}{y}+ t(\frac{1}{(y+xy^{n+1})^{n+1}}+
\frac{\epsilon}{y+xy^{n+1}})d(y+xy^{n+1}), 
\]
with $\epsilon=0$ or $1$. Setting $t=-1$, we get that 
$$\omega_{-1}=a(x,y)dx+b(x,y)dy\ \ \ \text{with}\ \ \ a(0)=-1,\ \text{and}\ b(0)=0.$$

Finally, replacing $\omega_0$ and $\omega_\infty$ by $\omega_0$ and $\omega_{-1}$,
we get back to the case $\Fo_0$ and $\Fo_\infty$ transversal in Lemma \ref{L:Pencil1}.
\end{proof} 
Here the proof of the theorem ends. 
\end{proof}

Now, we are going to solve the local classification problem in the remaining cases.

\begin{thm} \label{T:affinelogarithmic}
Consider a logarithmic affine structure on $S$ with polar divisor $D$
and let $D_0$ be a branch of $D$ with parabolic local monodromy.
Then, at a generic point of $D_0$, the affine structure is described 
by one of the model below:
\begin{enumerate}
\item  $\omega_t= dx -y^n\ln ydy+ ty^ndy$,  $n\in \Z_{\geq 0}$, 
\item $\omega_t= -\ln ydx+\frac{dy}{y^n}+(c-x)\frac{dy}{y}+tdx$, $n\in \Z_{>0}$, 
\end{enumerate}   
where $c=0$ or $1$. 
\end{thm}

\begin{proof}
The proof is quite similar to Theorem \ref{T:affinestructure}. 
Using the same argument we conclude that the structure
is defined by multivalued closed 1-forms $\omega_t=\omega_0+t\omega_{\infty}$.
 
Since the monodromy of the Riccati foliation $\Ho$ has an only fixed point, 
this means that the corresponding pencil of foliations $\Fo_t=\ker(\omega_t)$
has an only one element that extends through the polar divisor, 
that we can assume to be $\Fo_{\infty}$. At a generic point of $D_0$,
$\Fo_{\infty}$ is smooth, either tangent to $D_0$ (i.e. $D_0$ is a leaf),
or transversal to $D_0$. We discuss these two cases separately.

\begin{lemma}\label{L:Pencilinvariant} 
If  $D_0$ is invariant by $\Fo_\infty$, 
then there exist local coordinates such that the pencil is generated by 
\[
\omega_t= dx -y^n\ln ydy+ ty^ndy,\ \ \ \text{where}\ \ \ n\in \Z_{\ge0}.
\]
\end{lemma}

\begin{proof}
We choose coordinates such that $D_0=\{y=0\}$ and $\Fo_\infty=\ker(dy)$.
Therefore, we can write $\omega_{\infty}=f(y)dy$ and $\omega_0=gdx+hdy$
for multivalued functions $f,g,h$ (by closedness, $f$ only depends on $y$).
So, this pencil induces a Riccati foliation $\Ho$ given by
 \begin{equation*}
\omega=dz+(\frac{df}{f}-\frac{dg}{g})z+(\frac{h}{g}\frac{df}{f}-\frac{dh}{g})z^2.
\end{equation*}
Taking $\tilde z=\frac{1}{z}$ we have 
\begin{equation}\label{E:eq1}
\tilde\omega=d\tilde z+(\frac{dg}{g}-\frac{df}{f})\tilde z+(\frac{dh}{g}-\frac{h}{g}\frac{df}{f}).
\end{equation}
Since $\Ho$ is assumed to have a logarithmic pole, with parabolic monodromy, 
by Proposition \ref{P:gaugetransformation}, there exists an isomorphim $z\mapsto\varphi(x,y,z)$ 
of $\Proj^1$-bundle such that $\varphi_*\Ho$ is induced by 
$\hat\omega=d\hat z-n\hat z\frac{dy}{y} -y^{n-1}dy$, $n\in \Z_{\geq 0}$. 
The Riccati foliation has a single invariant section which is given by $\{\tilde z=\infty\}$
and $\{\hat z=\infty\}$ respectively, and we deduce that the bundle isomorphism takes the form
$$\hat z=a\tilde z+b,\ \ \ \text{where}\ \ \ a\in\Ol^*_0\ \ \ \text{and}\ \ \ b\in\Ol_0.$$
So $\Ho$ is induced by 
\begin{equation}\label{E:eq2}
\tilde \omega=d\tilde z +(\frac{da}{a}-n\frac{dy}{y})\tilde z +(\frac{db}{a}-\frac{nbdy}{ay}-\frac{y^{n-1}dy}{a}), \quad
n\in\Z_{\ge0}. 
\end{equation} 
Comparing equations (\ref{E:eq1}) and (\ref{E:eq2}) and then solving we have 
$g=cfay^{-n}$ and $h=cf(y^{-n}b-\ln y)+e$, where $c\in\C^*$ and $e\in \C$. We can suppose $c=1$
and $e=0$ by rescaling the $t$ parameter of the pencil, so 
$\omega_t=fay^{-n}dx+f(y^{-n}b-\ln y)dy+tfdy$. 
Closedness condition for $\omega_0$ gives:
\begin{equation}\label{E:eq3}
\frac{f'(y)}{f(y)}=\frac{n}{y}+\underbrace{\frac{b_x-a_y}{a}}_{\text{holomorphic}}
\end{equation}
Hence, applying Lemma \ref{L:multivaluedform}  with $\nu=n\ge0$,
we get a change of $y$-coordinate which normalizes $\omega_{\infty}=y^ndy$.
Replacing $f(y)=y^n$  in (\ref{E:eq3}), we get $a_y=b_x$ and therefore
$$\omega_t=\underbrace{adx+bdy}_{\text{closed}}-y^n\ln ydy+ty^ndy.$$
Since $a\in\Ol^*_0$, the $1$-form $adx+bdy$ can be normalized to 
$dx$ by a change of $x$-coordinate, which does not affect the other terms of $\omega_t$
and we get the form of the statement. 
\end{proof}

\begin{lemma}\label{L:Pencilnotinvariant} 
If  $D$ is transversal to $\Fo_\infty$, 
then there exist local coordinates such that the pencil is generated by 
\[
\omega_t=(1-x)\frac{dy}{y}-\ln ydx+tdx
\]
or 
\[
\omega_t= \frac{dy}{y^n}+(c-x)\frac{dy}{y}-\ln ydx+ tdx,\ \ \ n\in\Z_{>1}
\]
where $c=0$ or $1$.  
\end{lemma}

\begin{proof}
We choose coordinates such that $D_0=\{y=0\}$ and $\Fo_\infty=\ker(dx)$.
Since $\Fo_\infty$ is transversal to $D_0$, then $\omega_\infty$ extends through $D_0$,
and since it is closed, it can be normalized to $\omega_\infty=dx$.
We note that $\omega_\infty$ cannot vanish otherwise there is an extra polar component 
for the Riccati equation. We can write $\omega_0=gdx+hdy$
for multivalued functions $g,h$. So, the induced Riccati foliation $\Ho$ is given by
\begin{equation}\label{E:eq4}
\omega=dz+\frac{dh}{h}z+\frac{dg}{h}.
\end{equation}
Likely as in the previous proof, the fact that the monodromy is parabolic, $z=\infty$ is invariant
and the Riccati equation is logarithmic imply that we can write
\begin{equation}\label{E:eq5}
\omega=dz+(\frac{da}{a}-n\frac{dy}{y})z +(\frac{db}{a}-\frac{nbdy}{ay}-\frac{y^{n-1}dy}{a}), \quad
n\in\Z_{\ge0} 
\end{equation} 
where $a\in\Ol^*_0$  and $b\in\Ol_0$.
Comparing equations (\ref{E:eq4}) and (\ref{E:eq5}) and then solving we have 
$h=cay^{-n}$ and $g=cby^{-n}-c\ln y+e$, where $c\in\C$ and $e\in \C$. In fact,
one can check that $c\not=0$, otherwise the Riccati equation has a pole along $x=0$.
We can therefore suppose $c=1$ and $e=0$ by rescaling the $t$ parameter of the pencil, so that
$$\omega_t=a\frac{dy}{y^n}+\left(\frac{b}{y^{n}}-\ln y\right)dx+tdx.$$ 
Closedness condition for $\omega_0$ gives
\begin{equation}\label{E:eq6}
nb+y^n=y(b_y-a_x), \quad n\in \Z_{\geq 0}. 
\end{equation}
One easily check that $n=0$ yields a contradiction.
Now we are going to divide into two cases for the values of $n$. 

First case $n=1$. By Lemma \ref{L:multivaluedform} with parameter, there is a change of $y$-coordinate
that normalizes $a(x,y)\frac{dy}{y}$ to $a(x,0)\frac{dy}{y}$, i.e. we can now assume that $a=a(x)$.
Replacing it in (\ref{E:eq6}) and writing $b=y\tilde b$ we have $y{\tilde b}_y=a'(x)+1$. 
We deduce that $a'(x)+1={\tilde b}_y=0$, and therefore $\tilde b=\tilde b(x)$ and $a(x)=c-x$, $c\in\C$. 
However, $c=0$ is impossible, since this would create a pole along $x=0$ for the Riccati equation.
Applying the diffeomorphism $x\to cx$ and dividing $\omega_t$ by $c$,  
we can assume $c=1$ and hence 
$$\omega_t=(1-x)\frac{dy}{y}+(\tilde b(x)-\ln y)dx+tdx.$$
Finally taking the change of coordinates $(x,y)\to(x,y\exp(\frac{h}{x-1}))$, where 
$h'(x)=\tilde b(x)$, we can set $\tilde b=0$, and we get the first normal form of the statement.  

Second case $n>1$. By Lemma \ref{L:multivaluedform} with parameter, there is a change of $y$-coordinate
that normalizes  $a(x,y)\frac{dy}{y^n}=\frac{dy}{y^n}+\epsilon(x)\frac{dy}{y}$, $\epsilon\in\Ol_0$. 
Closedness condition (\ref{E:eq6}) shows that $b=y^n\tilde b$, and replacing it in (\ref{E:eq6}), 
we have $1+\epsilon'(x)=y\tilde b_y$. Hence $1+\epsilon'(x)={\tilde b}_y=0$, and
we get by integration $\tilde b=\tilde b(x)$ and $\epsilon=c-x$, for  some $c\in\C$. 
In the case $c\not=0$,  taking the diffeomorphism $x\to cx$ and rescaling $\omega_t$,
we can assume that $c=1$.  Therefore, we have 
$$\omega_t=\frac{dy}{y^n}+(c-x)\frac{dy}{y}+(\tilde b(x)-\ln y)dx+tdx,$$ 
where $c=0$ or $1$. Finally, we show that we can set $\tilde b(x)=0$ by a last change of $y$-coordinate.
For this, we consider the change of coordinate
$\psi(x,y)=(x,\varphi(x,y))$, where $\varphi=yu$, $u\in\Ol_0^*$. Here we choose $\varphi$ with the property that 
\[
\frac{\partial\varphi}{\partial y}\frac{dy}{\varphi^n}+(c-x)\frac{\partial\varphi}{\partial y}
\frac{dy}{\varphi}=\frac{dy}{y^n}+(c-x)\frac{dy}{y},
\]
this is equivalent to 
\[
\frac{1}{(1-n)y^{n-1}u^{n-1}}+(c-x)\ln u=\frac{1}{(1-n)y^{n-1}}-h(x).
\]  
For $h'(x)=\tilde b(x)$ the implicit function theorem guarantees the existence of a nonzero 
holomorphic solution $u(x,y)$. 
We can check that
\[
\psi^*\left(\frac{dy}{y^n}+(c-x)\frac{dy}{y}\right)=\frac{dy}{y^n}+(c-x)\frac{dy}{y}+\ln udx-h'(x)dx,
\]
Thus we have the desired conjugation.  
\end{proof}
So, the proof of the theorem ends. 
\end{proof}

Finally through Theorems \ref{T:affinestructure} and \ref{T:affinelogarithmic} 
we get normal forms for Ri\-cca\-ti foliations with logarithmic pole.

\begin{cor}\label{C:calmonlog}
Let $\Ho$ be a torsion-free Riccati foliation with logarithmic pole
over $\Proj^1\times (\C^2,0) \rightarrow (\C^2, 0)$, of polar divisor $D$. 
If the origin is a generic point of $D$, then after a change of coordinates on $(\C^2,0)$ we 
have the foliation $\Ho$ is given by one of the following 
equations with their respective monodromies listed bellow:  
\[\tiny
\begin{tabular}{|l|l|l|}
\cline{1-3}
pencil $\Fo_t=\ker(\omega_t)$  & Riccati $\Ho$ & Monodromy    \\
\cline{1-3}
$dx + ty^{\nu}dy$ & $\frac{dz}{z} + \nu\frac{dy}{y}$ & $z\mapsto e^{2\pi i\nu}z$    \\
\cline{1-3}
$dx + t(\frac{dy}{y^{n+1}}+\frac{dy}{y})$  & 
$\frac{dz}{z} - \frac{n+1+y^n}{y(1+y^n)}dy$ & identity    \\
\cline{1-3}
$dx -y^n\ln(y)dy+ ty^ndy$ & $dz-\left(n\frac{dy}{y}\right)z-y^{n-1}dy$  &
$z\mapsto z+1$\\
\cline{1-3}
$\frac{dy}{y^{n+1}}+(c-x)\frac{dy}{y}-\ln(y)dx+ tdx$  & 
$dz-\left( (n+1)\frac{dy}{y}+\frac{d(xy^n)-cny^{n-1}dy}{1+(c-x)y^n} \right)z-\frac{y^ndy}{1+(c-x)y^n}$ &
$z\mapsto z+1$\\
\cline{1-3}
\end{tabular}
\]
\end{cor}

\vskip1cm


\section{Logarithmic parallelizable webs}

A regular $d$-web on $(\C^2, 0)$ is the superposition  $\W= \Fo_1\boxtimes  \Fo_2 \boxtimes \cdots \boxtimes \Fo_d$ 
of $d$ regular foliations $\Fo_i$ that are moreover pairwise transversal:
\[
\Fo_i = \ker(\omega_i),\ \ \ \omega_i = a_i(x,y)dx+b_i(x,y)dy
\]
for $i=1, \cdots, d$, with 
$$\omega_i\wedge\omega_j=\det\begin{pmatrix} a_i&b_i\\ a_j&b_j\end{pmatrix}dx\wedge dy\not=0.$$
The $d$-web $\W$ is said {\bf parallelizable} if, after a convenient change of coordinates, 
it is defined by linear foliations of parallel lines, i.e. with $a_i,b_i$ constant functions.

It is well-known that regular $1$-webs (foliations), and $2$-webs are parallelizable, but 
a general $d$-web is not for $d\ge3$. The condition for a $3$-web to be parallelizable
is that it is {\bf hexagonal} (see \cite{Robert}). We will provide an explicit criterium in term of torsion later;
for what follows, it is enough to define hexagonal=parallelizable for a $3$-web.
A necessary condition for a $d$-web $\W$ to be linearizable is to require that all extracted $3$-webs 
are hexagonal; in that case, we will say that the $d$-web $\W$ is {\bf hexagonal}.
However, this is not enough, as can be shown by considering $d$ pencils of lines through 
a general collection of $d$ points.
A new constraint arising from $d\ge4$ is as follows.

Assume by a change of local coordinates that $b_i\not=0$ and denote the slope $e_i(x,y):=-\frac{a_i(x,y)}{b_i(x,y)}$.
Then the {\bf cross-ratio}:
\[
 (\Fo_1, \Fo_2; \Fo_3, \Fo_4) := \frac{(e_1-e_3)(e_2-e_4)}{(e_2-e_3)(e_1-e_4)}
\]
is a holomorphic function on $(\C^2, 0)$ intrinsically defined by $\W$. If a $4$-web is parallelizable,
then its cross-ratio must be constant. When we turn to a $d$-web $\W$, $d\ge4$, a necessary condition 
to be parallelizable is that any extracted $4$-web has constant cross-ratio; in that case, we will say that
the $d$-web $\W$ {\bf has constant cross-ratio}. By convention, a $3$-web has constant cross-ratio.
There is a natural link with the notion of pencil of foliations.

\begin{prop}\label{Prop:RegWebCriterium} Let $\W$ be a regular $d$-web, $d >3$. 
Then $\W$ is contained in a pencil of foliations $\{\Fo_t\}_{t\in \Proj^1}$   
if, and only if, $\W$ has constant cross-ratio. In that case,
are equivalent:
\begin{enumerate}
\item $\W$ is parallelizable,
\item $\W$ is hexagonal,
\item the pencil $\{\Fo_t\}_{t\in \Proj^1}$ is torsion free, i.e. $d\kappa=0$.
\end{enumerate}
\end{prop} 

Remind (Section \ref{ss:RiccatiVeronese}) that a pencil of foliations corresponds 
to a Riccati foliation $\Ho$ on $\Proj(T_S)$ via formula 
(\ref{eq:FirstIntegRiccati}-\ref{eq:FirstIntegToRicc}-\ref{eq:FirstIntegToPencil})
and the torsion $d\kappa$ is defined by (\ref{eq:LiftRicToCon}).
In fact, the torsion $d\kappa$ coincides (up to a non zero constant) 
to the Blaschke curvature of any extracted $3$-web of the pencil.

\begin{proof}
The first part is a consequence of a well-known property for Riccati foliations.
In fact, given a regular $3$-web $\W = \Fo_0\boxtimes \Fo_1\boxtimes \Fo_{\infty}$ on $(\C^2, 0)$, 
then there is
a unique pencil $\{\Fo_t\}_{t\in \Proj^1}$ that contains $\Fo_0, \Fo_1$ and $\Fo_{\infty}$ as 
elements. Precisely, $\Fo_t$ is defined as the unique foliation such that
\[
(\Fo_t, \Fo_0; \Fo_1, \Fo_{\infty})=t
\]
and we see that any extracted $d$-web must has constant cross-ratio.

For the second part, if $\Fo_0 \boxtimes \Fo_1 \boxtimes \Fo_{\infty}$ is parallelizable, say,
then in convenient coordinates we can assume that it is $\ker(dx)\boxtimes\ker(dx-dy)\boxtimes\ker(dy)$
and the pencil (defined by constant cross-ratio) is therefore parallel. 
Finally, recall that the pencil is parallelizable if, and only if, the torsion is zero, i.e. $d\kappa=0$.
\end{proof}

Globally, a regular $d$-web which has (locally) constant cross-ratio is better related
to a Veronese web (or a Riccati foliation on $\Proj(T_S)$ with possibly non trivial monodromy).

A {\bf singular $d$-web} $\W$ on a surface $S$ is given by an open covering $\mathcal{U} = \{U_i\}$ of $S$ and
$d$-symmetric 1-forms $\omega_i\in\Sym^d\Omega^1_S(U_i)$ such that 
for each non-empty intersection $U_i\cap U_j$ of elements of $\mathcal{U}$ there exists a non-vanishing
function $g_{ij}\in \Ol_S(U_i\cap U_j)$ such that $\omega_i = g_{ij}\omega_j$. The singular locus $\Delta$
is defined in charts by the discriminent of $\omega_i$ which is assumed to be non identically vanishing:
$\Delta$ is empty, or an hypersurface, and the $d$-web is regular on $S\setminus\Delta$.
We will say that $\W$ is hexagonal, or has constant cross-ratio, or is parallelizable if the property holds on $S\setminus\Delta$.

\begin{thm}\label{Thm:SingWebCriterium}
Let $\W$ be a singular $d$-web on $S$, $d\ge3$, with discriminant $\Delta$ and assume 
$\W$ has constant cross-ratio. Then, the Riccati foliation $\Ho$ defined by $\W\vert_{S\setminus\Delta}$ 
on $\Proj(T_{S\setminus\Delta})$ (see Proposition \ref{Prop:RegWebCriterium}) extends as a singular Riccati foliation on $\Proj(T_S)$ with regular-singularities along $\Delta$;  moreover, it has finite monodromy.

If $\W$ is hexagonal, then $\Ho$ is torsion-free and defines an affine structure on $S$
with regular-singularities along $\Delta$.
\end{thm}

\begin{proof} By definition of singular $d$-web, the local foliations defined by $\W$ 
at a generic point lift as a global multisection $\Sigma_\W$ of $\Proj(T_S)$, possibly ramifying over $\Delta$.
In local charts, sections of $\Sym^d\Omega^1_S$ take the form 
$$\omega=a_0(x,y)(dx)^d+a_{1}(x,y)(dx)^{d-1}(dy)+\cdots+a_d(x,y)(dy)^d$$
and the multisection $\Sigma_\W$ is defined by 
$$\Sigma_\W=\{a_0(x,y)+a_1(x,y)z+\cdots+a_d(x,y)z^d=0\},\ \ \ z=\frac{dy}{dx}.$$
The Riccati foliation $\Ho$ defined by $\W$ outside of $\Delta$ is determined as follows: 
the multisection $\Sigma_\W$ is a union of leaves for $\Ho$ and, since $d\ge3$, this is enough to define $\Ho$.
All other leaves are algebraic (defined by constant cross-ratio, see proof of Proposition \ref{Prop:RegWebCriterium})
and this implies that $\Ho$ extends as a singular Riccati foliation over $\Delta$.
The monodromy of $\Ho$ induces a permutation of the local branches of $\Sigma_\W$,
and we get a morphism
$$\mathrm{Mon}(\Ho)\to\mathrm{Sym}(d)$$
and it is injective since $d\ge3$; therefore, $\Ho$ has finite monodromy.
We claim that the extension is regular-singular along $\Delta$. Indeed, after ramified covering $\pi:\tilde S\to S$, 
ramifying over $\Delta$, we can assume that the monodromy of $\tilde\Ho=\pi^*\Ho$ is trivial,
and we get a global pencil (leaves define global meromorphic sections). But this implies that we can
trivialize the foliation by a global meromorphic gauge transformation. Therefore $\tilde\Ho$ is regular-singular.
But this implies that $\Ho$ is regular-singular, for instance using characterization in term of polynomial growth
(see \cite{Deligne}).

The last assertion is just a consequence of the definitions (Section \ref{Sec:MeroAffine}).
\end{proof}

Let $\W$ be a singular parallelizable $d$-web on $S$. 
Then $\W$ is said {\bf logarithmic} if the associate affine structure is logarithmic.

\begin{thm} Let $\W$ be a singular parallelizable $d$-web on $S$, $d\geq3$, 
with logarithmic singular points. Then, at a generic point of the discriminant $\Delta$,
$\W$ is contained in one of the following pencils:
\begin{enumerate}
\item $\{(dx)^q+ty^p(dy)^q=0\}_t$, with $(p,q)$ relatively prime positive integers,
\item $\{y^p(dx)^q+t(dy)^q=0\}_t$, with $(p,q)$ relatively prime positive integers,
\item $\{y^{n+1}dx+t(1+y^n)dy=0\}_t$, with $n$ positive integer.
\end{enumerate}
\end{thm}

\begin{proof}
By Theorem \ref{Thm:SingWebCriterium}, $\W$ defines a logarithmic affine structure on $S$
with poles along the discriminant $\Delta$, with finite monodromy.
Therefore, the monodromy cannot be parabolic, and we apply the local models of 
Theorem \ref{T:affinestructure}, with $\nu\in\mathbb Q^*$.
Then the normal form $dx+ty^\nu dy$ splits into $\nu=\frac{p}{q}$ or $-\frac{p}{q}$,
where $(p,q)$ relatively prime positive integers.
\end{proof}

We did not succeed to find a regular-singular version of Lemma \ref{L:Pencil4}
and this is what is missing to provide normal forms for general torsion-free $d$-webs.


\end{document}